\documentclass[11pt,reqno]{amsproc}

\title[Sobolev stability threshold for 2D shear flows near Couette]{The Sobolev stability threshold for 2D shear flows near Couette}

\author[J.~Bedrossian]{Jacob Bedrossian}
\address{Department of Mathematics, University of Maryland, College Park, MD 20742}
\email{jacob@cscamm.umd.edu}

\author[V.~Vicol]{Vlad Vicol}
\address{Department of Mathematics, Princeton University, Princeton, NJ 08544}
\email{vvicol@math.princeton.edu}

\author[F.~Wang]{Fei Wang}
\address{Department of Mathematics, University of Southern California, Los Angeles, CA 90089}
\email{wang828@usc.edu}

\usepackage[margin=1.25in]{geometry}
\usepackage{amsmath, amsthm, amssymb}
\usepackage{times, enumitem}
\usepackage{graphicx}

\usepackage{comment}

\usepackage[usenames,dvipsnames]{color}
\usepackage[colorlinks=true, pdfstartview=FitV, linkcolor=blue, citecolor=blue, urlcolor=blue]{hyperref}

\newtheorem{theorem}{Theorem}[section]

\newtheorem{proposition}[theorem]{Proposition}
\newtheorem{lemma}[theorem]{Lemma}
\theoremstyle{definition}

\newtheorem{remark}[theorem]{Remark}

\numberwithin{equation}{section}

\def\comma{ {\rm ,\qquad{}} }            
\def\curl{\mathop{\rm curl}\nolimits}    

\def\ne#1{#1_{\neq}}

\def\vv{\bar{v}}
\def\spacef{\ \ \ \ \ \ \ \ }

\def\RR{\mathbb R}
\def\TT{\mathbb T}

\def\eps{\varepsilon}
\def\ang#1{\langle #1\rangle}

\def\tilde{\widetilde}

\newcommand{\norm}[1]{\left\lVert#1\right\rVert}
\newcommand{\abs}[1]{\left\vert#1\right\vert}

\newcommand{\grad}{\nabla}
\newcommand{\brak}[1]{\langle #1 \rangle} 

\begin{document}

\begin{abstract}
We consider the 2D Navier-Stokes equation on $\mathbb T \times \mathbb R$, with initial datum that is $\varepsilon$-close in $H^N$ to a shear flow $(U(y),0)$, where $\| U(y) - y\|_{H^{N+4}} \ll 1$ and $N>1$. We prove that if $\varepsilon \ll \nu^{1/2}$, where $\nu$ denotes the inverse Reynolds number, then the solution of the Navier-Stokes equation remains $\varepsilon$-close in $H^1$ to $(e^{t \nu \partial_{yy}}U(y),0)$ for all $t>0$. Moreover, the solution converges to a decaying shear flow for times $t \gg \nu^{-1/3}$ by a mixing-enhanced dissipation effect, and experiences a transient growth of gradients. 
In particular, this shows that the stability threshold in finite regularity scales no worse than $\nu^{1/2}$ for 2D shear flows close to the Couette flow. \hfill \today.
\end{abstract}
\maketitle

\section{Introduction}

A fundamental problem in the field of hydrodynamic stability is to assess the stability of shear flows in the (2D or 3D) Navier-Stokes equations at high Reynolds number ($\nu \rightarrow 0$); see e.g. the texts \cite{DrazinReid81,SchmidHenningson2001,Yaglom12} and the references therein.

In this paper we consider the incompressible 2D Navier-Stokes equations
\begin{subequations}
  \begin{align}
  \partial_t \tilde v+ \tilde v\cdot\nabla \tilde  v-\nu\Delta  \tilde v+\nabla \tilde{p}&=0, \quad \nabla \cdot \tilde v = 0, \label{eq:NS}\\ 
  \tilde v(t=0) &= \tilde v_{\rm in} \label{eq:NS:IC}
  \end{align}
\end{subequations}  
on the domain $\Omega=\{(x,y) \in \mathbb{T}\times\mathbb{R}\}$, where $\nu$ denotes the inverse Reynolds number. Here $\mathbb{T}$ is the periodized interval $[0,1]$. 
The initial datum $\tilde v_{\rm in}$ is taken to be a small perturbation of a shear flow profile $(U(y), 0)$, and we write
\begin{align*}
\tilde v_{\rm in}(x,y)=(U(y), 0)+\vv_{\rm in}(x,y).
\end{align*}
If $\vv_{\rm in} \equiv 0$, the solution of the 2D Navier-Stokes equation \eqref{eq:NS}--\eqref{eq:NS:IC} is given by the heat evolution of the shear profile, i.e.
\begin{align}
(\bar{U}(t,y),0)  = (e^{\nu t \partial_{yy}} U(y), 0). \label{def:barU}
\end{align}
In this paper we are concerned only with shears $U(y)$ which satisfy $\norm{\partial_yU - 1}_{H^s} \ll 1$ for $s$ large enough (see Theorem \ref{thm:main} below for precise requirement) and hence are close to the Couette flow $U(y) = y$.  
For small (but nontrivial) $\vv_{\rm in}$, we look for a solution $\tilde v$ as a perturbation of this decaying shear profile
and define
\begin{align*}
\tilde v (t,x,y) =(\bar{U}(t,y), 0)+\vv(t,x,y)
\end{align*}
with the equations satisfied by $\vv$ being obtained from \eqref{eq:NS}--\eqref{eq:NS:IC} as
\begin{subequations}
\begin{align}
  \vv_t +\bar{U} \partial_x \vv+ \vv\cdot \nabla \vv -\nu\Delta \vv + (\bar{U}'\vv^y, 0) +\nabla p &=0, \quad   \nabla \cdot \vv =0,\label{EQ01} \\
  \vv(t=0)&=\vv_{\rm in}.
  \end{align}
\end{subequations}  
In \eqref{EQ01} we have used the notation $\bar{U}'=\partial_y\bar{U}$, and $\vv^y$ for the second, i.e. $y$, component of the vector field $\vv$.
Taking the curl of equation \eqref{EQ01} and denoting 
\begin{align*}
\omega=\curl \vv,
\end{align*} 
we obtain the vorticity formulation of the perturbed Navier-Stokes equation
\begin{subequations} \label{def:omegasys}
  \begin{align}
  \label{vort-orig}
  \omega_t +\bar{U}\partial_x \omega+ \vv\cdot \nabla \omega &=\bar{U}''\partial_x\psi+\nu\Delta \omega
  \\
  \label{pote-orig}
  \Delta \psi &=\omega
  \\
  \label{velo-orig}
  \vv&=\nabla^{\perp}\psi\\
  \omega(t=0) &= \omega_{\rm in},
  \label{vort-IC}
  \end{align} 
\end{subequations} 
where we have used the notation $\bar{U}''=\partial_{yy}\bar{U}$ and $\grad^\perp = (-\partial_y,\partial_x)$. In \eqref{pote-orig}  the stream function $\psi$ is taken to have zero mean over $\Omega$, so that we may write $\psi = - (-\Delta)^{-1} \omega$ and $\bar v =- \nabla^\perp (-\Delta)^{-1} \omega$.

\subsection{Statement and discussion of the main result}

The purpose of this paper is to study the long-time dynamics of the perturbation $\omega$ solving \eqref{def:omegasys} in the high Reynolds number limit $\nu \rightarrow 0$. 

Our main goal is to estimate the \emph{stability threshold} for the solution $\bar{U}(t,y)$, and in particular, to determine how it scales with respect to $\nu$. 
 That is, given an $N > 1$, try to find the \emph{smallest} $\gamma > 0$ such that $\norm{\omega_{\rm in}}_{H^N} = \eps \ll \nu^{\gamma}$ implies that the perturbation $\omega$ remains small in $L^2 \cap H^{-1}$, and to determine the dynamics of such stable solutions. 
We shall see that stability uniformly in $\nu$ in $H^s$ for any $s > 0$ is necessarily false as the solution undergoes a large transient growth in these norms due to the mixing caused by the shear. 
 
At high Reynolds number, in 3D experiments and computer simulations, nonlinear instability is often observed at a lower Reynolds numbers than what is predicted by linear theory. 
This is usually referred to as \emph{subcritical transition} \cite{SchmidHenningson2001,Yaglom12}.   
In some cases, such as \eqref{def:omegasys}, the flow is linearly stable at all Reynolds numbers, but the nonlinear stability threshold might be decreasing as $\nu \rightarrow 0$, resulting in instability at a finite Reynolds number in any experiment or simulation. 
Hence, \emph{given a norm} $\norm{\cdot}_X$, the goal is to determine a $\gamma = \gamma(X)$ such that (naturally we do not know a priori that it is a power law),  
\begin{align*}
\norm{\omega_{\rm in}}_{X} & \lesssim \nu^\gamma \quad \Rightarrow \quad \textup{stability} \\ 
\norm{\omega_{\rm in}}_{X} & \gg \nu^\gamma \quad \Rightarrow \quad \textup{possible instability}. 
\end{align*} 
In the applied math and physics literature, $\gamma$ is sometimes referred to as the \emph{transition threshold}. 

The minimal value of $\gamma$ is expected to depend non-trivially on the norm $X$. For example, in the numerical experiments of \cite{ReddySchmidEtAl98} on 3D Couette flow, it is estimated that ``rough'' initial perturbations (e.g. weaker $X$) result in a higher $\gamma$.  
In the case of \emph{3D} Couette flow, it was shown in \cite{BGM15I,BGM15II} that $\gamma = 1$ for $X$ taken as Gevrey-$m$ with $m < 2$, and that $\gamma \leq 3/2$ for $X = H^s$ for $s > 7/2$ in \cite{BedrossianGermainMasmoudi15a} (the latter estimate is consistent with the numerical estimation of $31/20$ given in \cite{ReddySchmidEtAl98}).

However, for the \emph{2D} Couette flow, it was shown in \cite{BMV14} that in fact $\gamma = 0$ for $X$ taken as Gevrey-$m$ with $m < 2$. That is, for initial perturbations taken sufficiently smooth, the Couette flow is uniformly stable at high Reynolds number and there is \emph{no} subcritical transition. It was shown earlier in \cite{BM13} that the Couette flow is also nonlinearly stable (in a suitable sense) for the 2D Euler equations, the case $\nu=0$, for such sufficiently smooth Gevrey perturbations.

In this paper, we estimate the stability threshold in 2D, now in {\em Sobolev regularity} (as opposed to in a Gevrey class~\cite{BMV14}). 
As in 3D, we could expect the stability threshold in Sobolev regularity to be worse than in a Gevrey class (or at least, our estimate of the stability threshold). 
Indeed, it follows from \cite{LinZeng11} for that topologies $X$ which are weaker than $H^{3/2}$, we should expect $\gamma>0$ even for the Couette flow.
{\em The stability threshold we estimate is $\gamma \leq 1/2$}. Our main result is:

\begin{theorem} \label{thm:main}
Let $N>1$, $0 < \nu \leq 1$, and $C\geq 1$ be a sufficiently large constant depending only on $N$ (in particular, it is independent of $\nu$). Consider a shear flow $U = U(y)$ such that  
\begin{align*}
  \Vert U'-1\Vert_{H^s(\RR)}+\Vert U''\Vert_{H^s(\RR)} = \delta  \leq C^{-1}
\end{align*}
for some $s \ge 2+N$ and $\delta$ independent of $\nu$. Assume that the initial perturbation obeys 
\begin{align}
\label{eq:thm:IC:size}
\Vert \omega_{\rm in}\Vert_{H^N(\Omega)} + \norm{\bar{v}_{\rm in}}_{L^2(\Omega)} =\varepsilon \leq C^{-1} \nu^{1/2}.
\end{align}
Then the global in time solution $\omega$ to \eqref{vort-orig}--\eqref{vort-IC} obeys
\begin{subequations} \label{ineq:thmineqs}
  \begin{align}
  \norm{\vv}_{L^\infty(0,\infty; L^2(\Omega))} + \norm{\omega_0}_{L^\infty(0,\infty;H^N(\Omega))} + \nu^{1/2}\norm{\grad \omega_0}_{L^2(0,\infty;H^N(\Omega))} & \leq C\varepsilon \\ 
  \norm{\omega \circ \left({x+ t\bar U(t,y)},y\right)}_{L^{\infty}( 0,\infty; H^{N}(\Omega))} 
  & \leq C \varepsilon \label{ineq:omthm}
  \end{align}
\end{subequations}
and we have the enhanced dissipation estimate
  \begin{align}
  \Vert \ne{\omega} \Vert_{L^{2}(0,\infty; H^{N}(\Omega)) }\leq C \varepsilon\nu^{-1/6} \label{ineq:thmED}.
  \end{align}
In \eqref{ineq:thmineqs}--\eqref{ineq:thmED} we have denoted by $\ne{\omega}$ the projection of $\omega$ onto its nonzero Fourier modes with respect to $x$, and by $\omega_0$ the projection of $\omega$ onto the zero Fourier mode with respect to $x$ (see also~\eqref{def:h0hneq}).  
\end{theorem}

\begin{remark} 
Estimate \eqref{ineq:thmED} encodes that the time-scale on which the deviation from the mean in $x$ is decaying is much {\em faster than the heat equation time-scale} (which would be $\nu^{-1/2}$). This ``enhanced dissipation'' effect is the key to this work, and is discussed further below in ~Section~\ref{sec:EDID}. 
\end{remark} 

\begin{remark} 
Notice that Theorem \ref{thm:main} is neither stronger nor weaker than the results of \cite{BMV14}. In \cite{BMV14}, much stronger regularity hypotheses are taken, however, $\eps$ can be chosen independently of $\nu$, whereas in Theorem \ref{thm:main}, the data is permitted to be much rougher but is instead required to be much smaller, specifically $\eps \lesssim \nu^{1/2}$. 
\end{remark}

\begin{remark} 
The composition in \eqref{ineq:omthm} looks slightly non-standard, but note that $\bar{U}(t,y)$ is essentially constant (in time) until $t \gtrsim \nu^{-1}$, at which point $\omega_{\neq}$ is essentially zero by \eqref{ineq:thmED}. 
\end{remark} 

\begin{remark}
The regularity assumption $N>1$ arises due to the fact that this is the smallest Sobolev exponent which guarantees the local in time well-posedness of the underlying 2D Euler equation ($\nu =0$). Since our constants are all  independent of $\nu$ (which may be taken arbitrarily small), such a requirement on the size of $N$ is arguably natural. 
\end{remark}

\begin{remark} 
Using techniques from \cite{BMV14}, one can slightly weaken the regularity requirement \eqref{eq:thm:IC:size} via local parabolic smoothing. 
Specifically, if one can decompose $\omega_{\rm in} = \omega_{\rm in}^R + \omega_{\rm in}^S$ such that for an $\eta > 0$ we have $\Vert \omega_{\rm in}^S \Vert_{H^N(\Omega)} + \nu^{1/2+\eta}\norm{\omega_{\rm in}^R}_{L^2} + \norm{\bar{v}_{\rm in}}_{L^2(\Omega)} =\varepsilon \ll_{\eta,N} \nu^{1/2}$, then the conclusions of Theorem \ref{thm:main} apply, except that the inequalities \eqref{ineq:thmineqs} and \eqref{ineq:thmED}  hold over $(1,\infty)$.  
\end{remark} 

\begin{remark}
Naturally, the value  $\gamma=1/2$ is much lower than for 3D Couette in Sobolev spaces (in 3D, the estimated value is $\gamma \leq 3/2$ \cite{BedrossianGermainMasmoudi15a}, consistent also with numerics \cite{ReddySchmidEtAl98}). The major differences between 2D and 3D are due to: the linearized 3D Couette flow induces vortex stretching and a 3D non-normal transient growth known as the lift-up effect; the class of $x$-independent flows is much larger in 3D; and the nonlinearity has a much more problematic ``nonlinear resonance'' structure in 3D (see \cite{BGM15I} for more discussion).   
It seems difficult to lower $\gamma$ here without confronting the main nonlinear source of regularity losses encountered in the case of Gevrey initial datum~\cite{BM13,BMV14}. See Subsection~\ref{sec:EDID} below for more discussion.  
\end{remark}

In order to fix the main ideas of the proof, we first prove Theorem~\ref{thm:main} in the case of the Couette flow $U(y) = y$ (see~Section~\ref{sec:Couette} below). 
The advantage here is that $\bar U(t,y) = y$ for all $t>0$, and thus the change of coordinates we use to account for the fast action of the shear flow is trivial (cf.~\eqref{def:CouetteChange}). Moreover, the case $U(y)=y$ has the advantage that the linearized Navier-Stokes equations lack the $\bar{U}''$ velocity term  (see e.g. the right side of \eqref{vort-orig}) and that the Biot-Savart law is a Fourier multiplier in the new variables. These make the proof of Theorem~\ref{thm:main} in the case of Couette more transparent. At the linear level, the main effects we take advantage of are inviscid damping and enhanced dissipation (see~Section~\ref{sec:EDID} below). 

In proving Theorem~\ref{thm:main} for the more general class of shear flows which are ``close to Couette'', we use the same main ideas as in the Couette case (e.g. the same norm), but we are faced with a number of technical difficulties due to a more complicated change of coordinates (cf.~\eqref{def:zv}) which needs to be adapted to the decaying shear. For example, we need to better quantify the transfer of kinetic energy from the non-zero modes (with respect to $x$) to the zero one. At the technical level, we also need to consider a number of new commutators between the Fourier multipliers that define the norm and multiplication operators with functions of $y$. These details are given in Section~\ref{sec:equivalence}.

\subsection{The linearized problem: inviscid damping and enhanced dissipation} \label{sec:EDID}
Before addressing the nonlinear problem, it is important to understand the properties of the linearized problem in the case $U(t,y) = y$, which we review here; see also \cite{BMV14} for similar discussions. 
In the case $U(t,y) = y$, the linearization of \eqref{def:omegasys} becomes
\begin{align*}
\partial_t \omega + y\partial_x \omega & = \nu \Delta \omega \\ 
\Delta \psi & = \omega.  
\end{align*}
These equations were first solved by Kelvin in \cite{Kelvin87}. The solution is given, in Fourier space, by 
\begin{align*}
\widehat{\omega}(t,k,\eta) &= \widehat{\omega}_{\rm in}(k,\eta + kt) \exp\left[-\nu \int_0^t k^2 + \abs{\eta - k(\tau - t)}^2 d\tau\right]. 
\end{align*}
From this formula, one verifies the enhanced dissipation effect: for some universal $c > 0$ there holds 
\begin{align*}
\norm{\omega_{\neq}(t)}_{L^2} \lesssim \norm{\omega_{\rm in}}_{L^2} e^{-c\nu t^3}, 
\end{align*}
which explains the accelerated time-scale \eqref{ineq:thmED}. 
Physically, the fast mixing of the Couette flow is sending information to high frequencies linearly in time, enhancing the viscous damping at the corresponding rate. 
This mixing-enhanced dissipation effect or related mechanisms have been studied in many works on linear equations, for example, in \cite{CKRZ08,GallagherGallayNier2009,Zlatos2010,BeckWayne11,VukadinovicEtAl2015,BCZGH15,BCZ15} and in the physics literature \cite{Lundgren82,RhinesYoung83,DubrulleNazarenko94,LatiniBernoff01,BernoffLingevitch94}. 
This effect implies that solutions strongly converge to a slowly decaying shear flow after times like $t \gtrsim \nu^{-1/3}$. 

A related, but more subtle effect, is that of inviscid damping, first noticed by Orr in \cite{Orr07}.  
For this, note that 
\begin{align*}
\widehat{\psi}(t,k,\eta) & = \frac{\widehat{\omega}_{\rm in}(k,\eta + kt)}{k^2 + \eta^2} \exp\left[-\nu \int_0^t k^2 + \abs{\eta - k(\tau - t)}^2 d\tau\right] \\ 
& = \frac{1}{(k^2 + \eta^2)\brak{\eta+kt}^2} \brak{\eta+kt}^2\widehat{\omega}_{\rm in}(k,\eta + kt) \exp\left[-\nu \int_0^t k^2 + \abs{\eta - k(\tau - t)}^2 d\tau\right], 
\end{align*}
which implies 
\begin{align}
\norm{\psi_{\neq}(t)}_{L^2} \lesssim \frac{1}{\brak{t}^2} \norm{\omega_{\rm in}}_{H^2} e^{-c\nu t^3}. \label{ineq:ID}
\end{align}
In particular, it follows that there is some decay which is \emph{independent of Reynolds number}, and indeed is present even in the case $\nu = 0$. 
Moreover, by similar arguments, one has 
\begin{align}
\norm{\partial_x\psi_{\neq}(t)}_{L^2} + \brak{t}^{-1}\norm{\partial_y \psi_{\neq}(t)}_{L^2} \lesssim \frac{1}{\brak{t}^2} \norm{\omega_{\rm in}}_{H^2} e^{-c\nu t^3}, \label{ineq:psineqLin}
\end{align}
which implies that the velocity field converges back to a shear flow even at infinite Reynolds number.
The name \emph{inviscid damping} is due to its relationship with Landau damping in plasma physics (see \cite{LinZeng11,Zillinger2014,BM13,WeiZhangZhao15} and the references therein for more discussion on inviscid damping and \cite{CagliotiMaffei98,Ryutov99,MouhotVillani11,BMM13} for discussions regarding Landau damping).
The regularity loss in \eqref{ineq:psineqLin} is physically meaningful and in particular is connected to a transient (non-normal) unmixing effect known as the Orr mechanism. 
When $\gamma$ is too small, this loss makes it impossible to naively close the kind of regularity estimates required to apply \eqref{ineq:ID} due to the large derivative loss in the right hand side (see \cite{BM13} for a more in-depth discussion). Physically, this manifests as the nonlinear echo resonance \cite{YuDriscoll02,YuDriscollONeil,MouhotVillani11,BM13}. 
In this paper we are able to balance this difficulty with the enhanced dissipation. 
This balance is what sets the requirement $\eps \ll \nu^{1/2}$ used here (see e.g. the treatment of estimate \eqref{ineq:CouetteTneq} below). 

 \subsection{Notations used in this paper}
Throughout this paper, we write \[ \ang{x}=\sqrt{1+x^2}.\] For a function $h(x,y)$, denote by 
\begin{align}
 h_0(y) =\int_{\TT} h(x,y)\,dx \quad \mbox{and} \quad h_{\neq}(x,y)=h(x,y)-h_0(y), \label{def:h0hneq}
\end{align}
the projections onto $0$ frequencies with respect to $x$ and the projection onto non-zero frequencies, respectively. 
By convention, we use $k, l$ to represent the Fourier variable of $x$, while $\xi, \eta$ are  the Fourier variable of $y$. The Fourier transform of a function $a$ is denoted by $\hat{a}(k,\xi)$.
We use the notation $f\lesssim g$ to express $f\le Cg$ for some constant $C>0$ that is independent of the parameters of interest and use $f \ll g$ when $f\le \varepsilon_0g$ for some {\em universal} constant $\varepsilon_0$ sufficiently close to $0$. 
The notation $L^pL^q = L^p_t L^q_{x,y}$ is used for the Banach space $L^p\left([0,T]; L^q(\Omega)\right)$ with norm
$
  \Vert f(t,x)\Vert_{L^pL^q}^p= \int_0^T\left(\int_{\Omega} |f|^q \,dx\right)^{p/q}\,dt
$
where $p,q\in[1,\infty]$ and $T>0$. We also use the space $L^pH^s = L^p_t H^s_{x,y}$ whose norm is given
$
  \Vert f\Vert_{L^pH^s}^p= \int_0^T\Vert f\Vert_{H^s}^{p}\,dt
$
for $p\in[1,\infty]$ and $s\ge0$.

\section{Stability threshold for the Couette flow} \label{sec:Couette}
In order to introduce several of the main ideas of the proof, we first prove Theorem \ref{thm:main} in the case that $U(y) = \bar{U}(t,y) = y$.  
We change coordinates to mod out by the fast mixing of the Couette flow, however, unlike the previous works on Navier-Stokes and Euler \cite{BM13,BMV14,BGM15I,BGM15II,BedrossianGermainMasmoudi15a}, we do not need a coordinate change that depends on the solution itself:  
\begin{subequations} \label{def:CouetteChange}
\begin{align}
z & = x-tv \\ 
v & = y.  
\end{align}
\end{subequations}
The change $v \mapsto y$ is made only to unify with Section~\ref{S4} below. 
If we then write $f(t,z,v) = \omega(t,z+tv,v)$, the Navier-Stokes equations \eqref{def:omegasys} become 
\begin{subequations} 
  \begin{align}
  \label{lnr-equ}
  &\partial_tf + u\cdot\nabla_Lf=\nu\Delta_Lf
  \\&
  \label{bs-law}
  u=- \nabla^{\perp}_L (-\Delta_L)^{-1}f, 
  \end{align}
\end{subequations}
where 
\begin{subequations} \label{def:gradL} 
\begin{align}
\grad_L & = \begin{pmatrix} \partial_z \\ \partial_v - t\partial_z \end{pmatrix} \\ 
\Delta_L & = \grad_L \cdot \grad_L = \partial_{zz} + (\partial_v - t\partial_z)^2.  
\end{align}
\end{subequations}
The following statement in the new coordinates then implies Theorem  \ref{thm:main} (in the case $U(y) = y$). 
\begin{theorem}
\label{T01}
Let $N>1$ and assume $\Vert f_{\rm in}\Vert_{H^N}=\varepsilon\ll\nu^{1/2}$ {with $0<\nu\le1$}. Then the unique global in time solution $f$ to \eqref{lnr-equ}--\eqref{bs-law} is such that
  \begin{align}
  \label{eq:T01:A}
  \Vert f\Vert_{L^{\infty}H^{N}} 
 +
  \nu^{1/2}\Vert \nabla_Lf\Vert_{L^2H^{N}} 
  \lesssim\varepsilon,
  \end{align}
and
  \begin{align}
  \label{eq:T01:B}
  \Vert \ne{f}\Vert_{L^{2} {H^{N}}}\lesssim \varepsilon\nu^{-1/6},
  \end{align}
 where the implicit constants do not depend on $\nu$ or on the initial datum.
\end{theorem}

The norm we employ is based on a special, time-dependent Fourier multiplier which is designed to capture transient unmixing effects in the linearization and its effect on the nonlinear problem. This particular technique was first introduced in \cite{BM13}, although the norm we use is more similar to the 3D Sobolev regularity work \cite{BedrossianGermainMasmoudi15a}. 
We specifically employ a Fourier multiplier $M$ which obeys the properties:
\begin{subequations} \label{eq:Mconds}
  \begin{align}
  M(0,k,\xi)&=M(t, 0, \xi)=1 \label{eq:M:cond:1}  \\
  1&\ge M(t, k, \xi)\ge c \label{eq:M:cond:2} \\
  -\frac{\dot{M}}{M}&\ge\frac{\abs{k}}{k^2+|\xi-kt|^2} \quad \mbox{for } k\neq0  \label{eq:M:cond:3}\\
  \left|\frac{\partial_{\xi}M(k,\xi)}{M(k,\xi)}\right|  & \lesssim \frac{1}{|k|} \quad \mbox{for } k\neq0, \mbox{ uniformly in } \xi \label{eq:M:cond:4}\\
  1&\lesssim  \nu^{-1/6}\left(\sqrt{-\dot{M}M(t,k,\eta)}+\nu^{1/2}|k,\eta-kt|\right) \quad \mbox{for } k\neq0 \label{eq:M:cond:5} \\ 
  \sqrt{-\dot{M}M(t,k,\eta)} & \lesssim \brak{\eta-\xi} \sqrt{-\dot{M}M(t,k,\xi)}, \label{eq:M:cond:6}
  \end{align}
\end{subequations}
for a constant  $c \in (0,1)$  which is independent of $\nu$. 
The construction of such a multiplier $M$  is given in Appendix~\ref{sec:app} below and is similar to one used in \cite{BedrossianGermainMasmoudi15a} (some aspects of the multiplier also appeared in \cite{Zillinger2014}). See Lemma \ref{l01} in Appendix \ref{sec:app} for details. 

\begin{proof}[Proof of Theorem~\ref{T01}]
Defining 
\begin{align*} 
A=M\ang{D}^N,
\end{align*} 
where $\ang{D}=\sqrt{1+D^2}$, it follows from Plancherel's inequality and \eqref{eq:M:cond:2} that 
\begin{align}
\norm{f}_{H^N} \lesssim_{c} \| A(t) f\|_{L^2} \leq \|f\|_{H^N},\label{ineq:Aequiv}
\end{align}
and that $\norm{A(0)f}_2 = \norm{f}_{H^N}$.
It is important to note that such a simple equivalence to a standard norm (uniform in $\nu$ and $t$) has not been true in most previous works using similar Fourier multiplier-based techniques \cite{BM13,BMV14,BGM15I,BGM15II,BedrossianGermainMasmoudi15a}; the exception being \cite{Zillinger2014}. 
This distinction is one of the primary reasons that the proof of Theorem \ref{thm:main} is significantly less technical than the previous works (it also makes the technique essentially a Fourier-side analogue of Alinhac's ghost energy method \cite{Alinhac01}; see \cite{BedrossianGermainMasmoudi15a} for more discussion).  

By \eqref{ineq:Aequiv}, the proof of Theorem~\ref{T01} is then based on establishing the a priori estimate
\begin{align}
  \Vert Af\Vert_{L^{\infty}L^{2}} 
 +
  \nu^{1/2}\Vert \nabla_LAf\Vert_{L^2L^{2}} 
 +
  \left\Vert \sqrt{-\dot{M}M}\ang{D}^Nf\right\Vert_{L^2L^{2}}
  \leq 8 \varepsilon.
  \label{eq:T01:*}
  \end{align}
In turn, \eqref{eq:T01:*} implies \eqref{eq:T01:A} since $\| A f\|_{L^2}$ and $\|f\|_{H^N}$ are equivalent. Moreover, \eqref{eq:T01:B} follows from \eqref{eq:T01:*}, since \eqref{eq:M:cond:5} and \eqref{ineq:Aequiv} imply that for any $f$, 
\begin{align}
\norm{f_{\neq}}_{L^2H^N} & \lesssim \nu^{-1/6}\left(\nu^{1/2}\Vert \nabla_LAf_{\neq}\Vert_{L^2L^{2}} + \left\Vert \sqrt{-\dot{M} M}\ang{D}^N f\right\Vert_{L^2L^{2}}\right). \label{ineq:fneqdec}
\end{align}

By the local well-posedness in $H^N$ of the $2$D Navier-Stokes equations, for $t$ sufficiently small, there holds
\begin{align*}
\Vert Af\Vert_{L^{\infty}(0,t;L^{2})} + \nu^{1/2}\Vert \nabla_LAf\Vert_{L^2(0,t;L^{2})} + \left\Vert \sqrt{-\dot{M}M}\ang{D}^Nf\right\Vert_{L^2(0,t;L^{2})} < 2 \varepsilon, 
\end{align*} 
and all the quantities on the right hand side of \eqref{eq:T01:*} take values continuously in time. 
We will use a bootstrap argument to extend these estimates for all time. 
Hence, define  $T \leq \infty$ to be the maximal time such that
  \begin{align}
  \label{}
  \Vert Af\Vert_{L^{\infty}(0,T;L^{2})} + \nu^{1/2}\Vert \nabla_LAf\Vert_{L^2(0,T;L^{2})} + \left\Vert \sqrt{-\dot{M}M}\ang{D}^Nf\right\Vert_{L^2(0,T;L^{2})} \le8\varepsilon
  \label{eq:T01:**}
  \end{align}
holds; by the above discussion, $T > t$.  We refer to this inequality as the {\em bootstrap assumption}. We next prove that in fact \eqref{eq:T01:**} holds with the constant $8$ on the right side replaced by $4$, implying the global stability (i.e. that $T$ may be taken arbitrarily large).
\begin{proposition} \label{prop:CouetteBoot}
For all $N > 1$ and $\varepsilon$ sufficiently small (depending only on $N$), inequality \eqref{eq:T01:**} holds with the ``8'' replaced with a ``4'' on $[0,T)$, and hence by continuity, $T = +\infty$.  
\end{proposition}

First, applying the operator $A$ to \eqref{lnr-equ}, we obtain
the energy estimate for $f$ which reads
  \begin{align*}
 &\frac{1}{2}\frac{d}{dt}\Vert Af(t)\Vert_{L^{2}}^2 
 +
  \nu\Vert \nabla_LAf(t)\Vert_{L^{2}}^2 
 +
  \left\Vert \sqrt{-\dot{M}M}\ang{D}^Nf(t)\right\Vert_{L^{2}}^2
  \nonumber\\&\qquad
  =
  -\int A(u(t)\cdot\nabla_Lf(t))Af(t) \,dV
  \end{align*}
where we use the notation $dV = dz\, dv$.
Integrating the above equality in time on the interval [0, T], we obtain (where we have abbreviated the time-interval from norms in order to reduce clutter)
  \begin{align*}
  &\frac{1}{2}\Vert Af(T)\Vert_{L^{2}}^2 
 +
  \nu\Vert \nabla_LAf\Vert_{L^2L^{2}}^2 
 +
  \left\Vert \sqrt{-\dot{M}M}\ang{D}^Nf\right\Vert_{L^2L^{2}}^2
  \nonumber\\&\qquad
  =
  \frac{1}{2}\Vert Af(0)\Vert_{L^{2}}^2 
  -\iint A(u\cdot\nabla_Lf)Af\,dVdt
  =\frac{1}{2}\Vert Af(0)\Vert_{L^{2}}^2-\mathcal{T}
  \end{align*}
where $\mathcal{T}=\iint A(u\cdot\nabla_Lf)Af\,dVdt.$
In order to estimate $\mathcal{T}$, which is the bulk of the proof, we decompose the velocity $u$ into the zero frequency and non-zero frequency parts as
  \begin{align}
  \label{}
  \mathcal{T}
  =
  \iint A(u_0^x\partial_z f)Af\,dVdt + \iint A(\ne{u}\cdot\nabla_Lf)Af\,dVdt
  =
  \mathcal{T}_0+\mathcal{T}_{\neq}.
  \end{align}
  
We first bound the term $\mathcal{T}_{\neq}$ using \eqref{eq:M:cond:2} and the algebra property of $H^N$ as
  \begin{align}
  \label{EQ301}
  |\mathcal{T}_{\neq}|
  &=
  \left| \iint A(\nabla_L^{\perp}\Delta_L^{-1}\ne{f}\cdot\nabla_Lf)Af\,dVdt \right| \notag\\
  &\lesssim
  \Vert \nabla_L^{\perp}\Delta_L^{-1}\ne{f}\Vert_{{L^2}H^{N}}\Vert \nabla_Lf\Vert_{L^{2}H^{N}}\Vert Af\Vert_{L^{\infty}L^{2}}.
  \end{align}
 From \eqref{EQ301}, the property \eqref{eq:M:cond:3} of $M$, and the fact that $\ne{f}$ contains only modes $|k| \geq 1$, we obtain the following, using the  bootstrap assumption \eqref{eq:T01:**}, 
  \begin{align}
  |\mathcal{T}_{\neq}|
  &\lesssim
  \left\Vert \sqrt{-\dot{M}M}\ang{D}^N\ne{f}\right\Vert_{L^{2}L^{2}}\Vert \nabla_Lf\Vert_{L^{2}H^{N}}\Vert Af\Vert_{L^{\infty}L^{2}}
  \nonumber\\&
  \lesssim \varepsilon^3\nu^{-1/2}. \label{ineq:CouetteTneq}
  \end{align}
This is consistent with Proposition \ref{prop:CouetteBoot} by choosing $\varepsilon \ll \nu^{1/2}$ (where the implicit constant depends only on $N$). This term, $\mathcal{T}_{\neq}$, is the only {\em critical} one for which we need the assumption $\varepsilon\ll\nu^{1/2}$. 

In order to estimate the zero frequency part, we note that $f_0$ is independent of $z$ and thus we may rewrite $\mathcal{T}_0$ as
  \begin{align}
  \label{EQ302}
  \mathcal{T}_0 
  & =  \iint A(-\partial_v\partial_v^{-2}f_0\, \partial_zf)Af\,dVdt = \iint A(-\partial_v\partial_v^{-2}f_0\, \partial_z\ne{f})A\ne{f}\,dVdt. 
  \end{align}
where we also used, by Plancherel's theorem, 
\begin{align*}
\int A(-\partial_v\partial_v^{-2}f_0\partial_z\ne{f})Af_0\,dV=0.
\end{align*}
Additionally, using the cancellation
\begin{align*}  
\int \partial_v\partial_v^{-2}f_0\partial_zA\ne{f}A\ne{f}\,dV =0,
\end{align*}
we obtain from \eqref{EQ302} that
\begin{align}
  \mathcal{T}_0 = \iint \Big(A(-\partial_v\partial_v^{-2}f_0\, \partial_z\ne{f})+\partial_v\partial_v^{-2}f_0\, \partial_zA\ne{f}\Big )A\ne{f}\,dVdt. \label{eq:T0Couette}
\end{align}
By Plancherel's theorem, 
  \begin{align}
  \label{EQ303}
 - \mathcal{T}_0
  =
 \frac{1}{4 \pi^2}\sum_{k\neq0}\int \!\!\! \int\!\!\! \int  \left(A(k,\xi)-A(k,\xi-\eta)\right)\frac{\eta}{\eta^2}\hat{f}(0,\eta)k\hat{f}(k,\xi-\eta)
  A(k,\xi)\bar{\hat{f}}(k,\xi)\,d\eta d\xi dt.
  \end{align}
We decompose the difference (a commutator in real variables) $A(k,\xi)-A(k,\xi-\eta)$ into two commutators, one due to $M$ and one due to $\langle D\rangle^N$: 
  \begin{align*}
  A(k,\xi)-A(k,\xi-\eta)
  &=
  M(k,\xi)\left( (1+k^2+\xi^2)^{N/2}-(1+k^2+(\xi-\eta)^2)^{N/2}\right)
 \nonumber\\&\qquad 
  +
  (M(k,\xi)-M(k,\xi-\eta))(1+k^2+(\xi-\eta)^2)^{N/2}
  \nonumber\\&
  =
  {\rm com_1+com_2}.
  \end{align*}
For ${\rm com_1}$, we use the mean value theorem to obtain, for some $\theta \in [0,1]$, 
  \begin{align}
  |{\rm com}_1|
  &=
  \left|M(k,\xi)\frac{N}{2}(1+k^2+(\xi-\theta\eta)^2)^{N/2-1})2(\xi-\theta\eta)\eta\right|
  \nonumber\\&
  \lesssim
  \left((1+k^2+(\xi-\eta)^2)^{(N-1)/2}+(1+k^2+\xi^2)^{(N-1)/2}\right)|\eta|.
    \label{eq:com:1}
  \end{align}
To estimate ${\rm com_2}$, we additionally recall that the multiplier $M$  satisfies \eqref{eq:M:cond:4}.
From \eqref{eq:M:cond:4} and the mean value theorem we deduce
  \begin{align}
  \label{}
  |{\rm com_2}|
  \lesssim
  \frac{|\eta|}{|k|}(1+k^2+(\xi-\eta)^2)^{N/2}.
  \label{eq:com:2}
  \end{align}
Combining \eqref{EQ303} with the bounds \eqref{eq:com:1} and \eqref{eq:com:2} for ${\rm com_1}$ and ${\rm com_2}$ respectively, we arrive at
  \begin{align*}
  |\mathcal{T}_0|
  &\lesssim
  \sum_{k\neq0}\int \!\!\! \int\!\!\! \int \left((1+k^2+(\xi-\eta)^2)^{N/2}+(1+k^2+\xi^2)^{N/2}\right)
  \nonumber\\&\qquad \qquad 
  \times |\hat{f}(0,\eta)\hat{f}(k,\xi-\eta)|
  |A(k,\xi)\bar{\hat{f}}(k,\xi)|\,d\eta d\xi dt.
  \end{align*}
By Young's convolution inequality, and $N> 1$, we obtain, 
  \begin{align} 
  |\mathcal{T}_0| \lesssim \Vert f_0\Vert_{L^{\infty}H^{N}}\Vert f_{\neq}\Vert_{L^{2}H^{N}}^2 .\label{ineq:T0postCommCouette}
  \end{align}
By \eqref{ineq:fneqdec} and the bootstrap assumption \eqref{eq:T01:**}, we obtain
  \begin{align*}
  |\mathcal{T}_0| &\lesssim \|f_0\|_{L^\infty H^N} \nu^{-1/3} \left( \| \sqrt{-\dot{M}M} \langle D\rangle^N f_{\neq} \|_{L^2 L^2}+ \nu^{1/2} \| \nabla_L A f_{\neq}\|_{L^2 L^2} \right)^2\notag\\
  &\lesssim \varepsilon^3\nu^{-1/3},   
  \end{align*}
which is sufficient to deduce Proposition \ref{prop:CouetteBoot} under the hypotheses $\varepsilon \ll \nu^{-1/3}$.
This completes the proof of Proposition \ref{prop:CouetteBoot} and hence of Theorem \ref{T01}.    
\end{proof}

\section{Shear flows close to Couette}
\label{S4}
In this section we consider the more general class of shear flows with initial data $(U(y),0)$ which are sufficiently close to Couette, in the sense that
  \begin{align}
  \label{EQ400}
  \Vert U'-1\Vert_{H^s}+\Vert U''\Vert_{H^s}\le \delta\ll 1
  \end{align}
for some $s\ge2+N$. 
  \begin{remark}
  Recall from Theorem \ref{thm:main} that the constant $\delta$ is independent of $\varepsilon$ and $\nu$.
  \end{remark}
Recall the definition of $\bar{U}$ from \eqref{def:barU}, the solution of the Navier-Stokes equations with the initial datum $(U(y),0)$.   
From standard estimates on the heat equation there holds, 
\begin{subequations} 
 \begin{align}
  \label{}
 \sup_{t>0} \Vert \bar{U}'(t,\cdot)-1\Vert_{H^s}
 & \le
  \Vert U'-1\Vert_{H^s} \\ 
 \sup_{t>0} \Vert \bar{U}''(t,\cdot) \Vert_{H^s}
  & \le
  \Vert U''\Vert_{H^s} \\ 
  \Vert \bar{U}''\Vert_{L^2_tH^s_y}  & \lesssim \delta  \nu^{-1/2}. \label{ineq:Uppdecay}
  \end{align}
\end{subequations}

\subsection{Coordinate System} 
Recall that $\bar{U}(t) = e^{\nu t \partial_{yy}}U$ is the decaying background shear flow.  
Instead of \eqref{def:CouetteChange} we need to work with a coordinate change adapted to this shear flow, and thus consider 
\begin{subequations} \label{def:zv} 
\begin{align}
x \mapsto z & = x - t\bar{U}(t,y) \\ 
y \mapsto v & = \bar{U}(t,y). 
\end{align}
\end{subequations}
For $\delta$ sufficiently small, the map $(x,y) \mapsto (x,v)$ is invertible; see ~Subsection~\ref{sec:equivalence} below.
The choice $y \mapsto v$ is made so that $\partial_y \mapsto \partial_y\bar{U}\left(\partial_v - t\partial_z\right)$ (see \eqref{deri:1} below). 
This ensures that the critical times are not significantly perturbed by the coefficients, and featured in all of the previous works on Navier-Stokes or Euler near Couette flow, introduced first in \cite{BM13}.  
The change from $x \mapsto z$ is essentially rewinding by the characteristics of the shear flow, although in way which is also well-adjusted for the requirement on $v$ (it is also well-adjusted for the parabolic decay of $\bar{U}(t,y)$, although this is less obvious). 

In the new coordinate system we define quantities corresponding to vorticity, stream function, and  velocity, respectively, by
  \begin{align}
  \label{trans:cdnt}
  f(t, z(t,x,y), v(t,y)) &=\omega(t, x, y)
  \nonumber\\
  \phi(t, z(t,x,y), v(t,y)) &=\psi(t, x, y)
  \nonumber\\
  u(t,z(t,x,y),v(t,y)) & =\vv(t,x,y).
\end{align}
It is also convenient to denote the spatial derivatives of the shear flow in the new coordinates
\begin{subequations} \label{def:ab}
\begin{align}
  a(t,v(t,y))&=\bar{U}'(t,y) \\
  b(t,v(t,y))&=\bar{U}''(t,y).
\end{align} 
\end{subequations}
For any function $\tilde h$ in the $(x,y)$ coordinates, the corresponding function $h$ in the $(z,v)$ coordinates 
\begin{align*}
 h(t,z,v)=\tilde h(t,x,y),
\end{align*} 
the differential operator $\nabla \mapsto \nabla_t$, is defined by
\begin{align}
  \label{deri:1}
  \nabla \tilde{h}(t, x, y)
  =
    \begin{pmatrix}
    \partial_x \tilde h\\
    \partial_y \tilde h
    \end{pmatrix}
  =
    \begin{pmatrix}
    \partial_z h\\
    a(\partial_v-t\partial_z) h
    \end{pmatrix}
   =
    \begin{pmatrix}
    \partial^t_z h\\
    \partial^t_v h
    \end{pmatrix}
   =\nabla_t h, 
\end{align} 
and analogously $\Delta \mapsto \Delta_t$ is defined via, 
\begin{align}
\Delta \tilde{h}(t,x,y) = (\partial_z^2+a^2\partial^L_{vv}+b\partial^L_{v})h = \Delta_t h. \label{def:Deltat}
\end{align}
Finally, the time derivative maps to the following, using $\partial_t \bar{U} = \nu \partial_{yy}\bar{U}$ and  the definitions of $a,b$ from \eqref{def:ab}, 
\begin{align}
\partial_t\tilde{h} & = \partial_t h + \partial_z h \left(-\bar{U} - t \partial_t \bar{U}\right) + \partial_v h \partial_t \bar{U} \nonumber \\  
& = \partial_t h - v \partial_z h + \nu b(\partial_v - t\partial_z)h. \label{def:dth}
\end{align}
Recall that \eqref{def:gradL} is the ``linear part'' of the operator $\nabla_t$. 
The relationship between $f$ and $\phi$ is given through \eqref{def:Deltat}: 
  \begin{align}
  \label{deri:3}
  f&=\Delta_t \phi=(\partial_z^2+a^2\partial^L_{vv}+b\partial^L_{v})\phi
  =(\partial_z^2+a^2(\partial_{v}-t\partial_z)^2+b(\partial_{v}-t\partial_z))\phi
  \nonumber\\&
  =\Delta_L \phi + ((a^2-1)\partial^L_{vv}+b\partial^L_{v})\phi.
  \end{align}
Note that by the chain rule, we have
  \begin{align}
  \label{g:b}
  b=a \partial_v a.    
  \end{align}
Using the above notations, by applying \eqref{def:gradL}, \eqref{def:Deltat}, and \eqref{def:dth} we find that $f$ satisfies the equations
\begin{subequations} \label{eq:fsys}  
\begin{align}
  &\partial_tf + u\cdot\nabla_tf=b\partial_z\phi+\nu\tilde{\Delta}_tf
  \label{vort}\\&
  \Delta_t\phi=f
  \label{pote}\\&
  u =  \nabla^{\perp}_t\phi
  \label{velo},
  \end{align}
\end{subequations}
 where the modified Laplace operator is given by 
 \begin{align*}
 \tilde{\Delta}_t=\partial_{z}^2+a^2\partial_{vv}^L = \Delta_L + (a^2 -1) \partial_{vv}^L.
 \end{align*} 
 In particular, note the cancellation between the lower order term in $\Delta_t$
and the last term in \eqref{def:dth}. 
 
\begin{remark} 
Recall the notations in \eqref{def:h0hneq}. As $\bar{U}$ is independent of $x$, there holds
\begin{align*}  
(\nabla_t {h})_0=\nabla_t{h}_0\comma
  \ne{(\nabla_t {h})}=\nabla_t \ne{{h}}.
\end{align*}
As a result, from the divergence free condition we deduce that $u_0^y = 0$. 
\end{remark}

\subsection{Equivalence of the two coordinate systems}
\label{sec:equivalence}
We shall prove Theorem \ref{thm:main} in the new coordinate system \eqref{def:zv}. Here we discuss how to relate this coordinate system to the original $(x,y)$. 
The first lemma, the proof of which can be found in \cite{InciKappelerTopalov13}, provides a composition inequality in fractional Sobolev spaces. 
\begin{lemma}[Fractional Sobolev Composition]
\label{diff:back}
Let $s'>2$, $s'\ge s\ge0$, $f\in H^{s}(\mathbb{R})$, and $g\in H^{s'}(\mathbb{R})$ be such that $\Vert g\Vert_{H^{s'}} \le \delta$. Then, there holds
\begin{align*}
  \Vert f\circ(I+g)\Vert_{H^s} \leq C_{s,s'}(\delta) \Vert f\Vert_{H^s},
\end{align*}
where the implicit constant obeys $C_{s,s'}(\delta) \to 1$ as $\delta\to 0$.
\end{lemma}

The next implicit function theorem is important for inverting the coordinate transformation \eqref{def:zv} and carrying the information back to the original coordinates. 
\begin{lemma}[Implicit Function Theorem]
\label{implicit}
Let $s>2$. There exists an $\varepsilon_0=\varepsilon_0(s)$ such that if $\Vert \alpha\Vert_{H^s}\le\varepsilon_0$, then there is a unique solution $\beta$ to 
\begin{align*}
  \beta(y)=\alpha(y+\beta(y)),
\end{align*}
with $\Vert \beta\Vert_{H^s}\lesssim\varepsilon_0$.
\end{lemma}
The next lemma shows that we are allowed to take all the information in the original system to the new system.
\begin{lemma}
\label{diff:to}
Let $s'>2$, $s'\ge s\ge0$, $f\in H^{s}(\mathbb{R}^2)$, and $g\in H^{s'}(\mathbb{R}^2)$ be such that $\Vert g\Vert_{H^{s'}} \le \delta$. Then, there holds
\begin{align*}
  \Vert f\Vert_{H^s} \leq C_{s,s'}(\delta) \Vert f\circ(I+g)\Vert_{H^s},
\end{align*}
where the implicit constant obeys $C_{s,s'}(\delta) \to 1$  as $\delta \to 0$.
\end{lemma}
\begin{proof}[Proof of Lemma~\ref{diff:to}]
Let $h=(I+g)(y)=y+g(y)$. Then we rewrite this equality as
$
h-y=g(h-y+h).
$
Now thinking of $y$ as a function of $h$, we denote $\beta(h)=h-y$ and deduce
$
\beta(h)=g(h+\beta(h)).
$
By Lemma \ref{implicit}, there exists a $\beta$ such that the above equality holds and 
furthermore
\begin{align*}
\Vert \beta\Vert_{H^s}\lesssim\delta.
\end{align*}
Now we write $f$ as 
\begin{align*}  
f(h)=f\circ(I+g)\circ(I+\beta)(h)
\end{align*}
and use Lemma \ref{diff:back} to conclude the proof.
\end{proof}

From the properties of the heat equation and Lemma \ref{diff:to}, we can deduce the following lemma regarding the coefficients $a$ and $b$, which arise in \eqref{eq:fsys}. 
\begin{lemma} 
From \eqref{g:b}, Lemma \ref{diff:to}, and \eqref{ineq:Uppdecay}, for $\delta$ sufficiently small, there holds for $\sigma>2$, 
  \begin{align}
  \label{l2:g}
  \Vert \partial_v{(a^2-1)}\Vert_{L^2_tH^\sigma_y}
  =\Vert 2b\Vert_{L^2_tH^\sigma_y}
  \lesssim
  \Vert \bar{U}''\Vert_{L^2_tH^\sigma_y}
  \lesssim
  \delta  \nu^{-1/2},
  \end{align}
where the implicit constant does not depend on $\delta$ and $\nu$. 
Similarly, by \eqref{trans:cdnt}, \eqref{EQ400}, and  Lemma \ref{diff:to}, we have that 
  \begin{align}
  \label{gb:bd}
  \Vert a-1\Vert_{H^\sigma}
 +
  \Vert b\Vert_{H^\sigma}
  \le
  2(\Vert \bar{U}'-1\Vert_{H^\sigma}
  +
  \Vert \bar{U}''\Vert_{H^\sigma})
  \le 2\delta
  \end{align}
holds for $\sigma > 2$ and $\delta>0$ {small enough}.
\end{lemma}

\subsection{Main result}
The main result of this section, which in particular implies Theorem~\ref{thm:main}, is:

\begin{theorem}
\label{T02}
Let $N>1$ and assume that the shear flow $(U(y),0)$ satisfies \eqref{EQ400}. If the initial vorticity perturbation obeys $\Vert f_{\rm in}\Vert_{H^N}=\varepsilon\ll\nu^{1/2}$, then the {unique} solution $f$ to \eqref{vort}--\eqref{velo} obeys the global in time estimates
\begin{subequations} \label{ineq:estT02}
  \begin{align}
  & \Vert f\Vert_{L^{\infty}H^{N}} + \nu^{1/2}\Vert \nabla_Lf\Vert_{L^2H^{N}} \lesssim\varepsilon, \\
 & \Vert u_0^x \Vert_{L^\infty L^{2}} + \nu^{1/2}\Vert \partial_v u_0^x\Vert_{L^2L^{2}} \lesssim\varepsilon,
  \end{align}
\end{subequations}
and
  \begin{align}
  \Vert \ne{f}\Vert_{L^{2}H^{N}} \lesssim \varepsilon\nu^{-1/6}\label{ineq:estfne}
  \end{align}
where the implicit constants do not depend on $\nu$ or on the initial datum.
\end{theorem}
Indeed, the estimates \eqref{ineq:estT02} and \eqref{ineq:estfne} together with Lemmas \ref{diff:back} and \ref{diff:to}, imply Theorem \ref{thm:main}.

\subsection{Proof of Theorem \ref{T02}}  
The proof of Theorem \ref{T02} uses the main ideas of Theorem \ref{T01}, however, due to the more complicated change of coordinates \eqref{def:zv}, we need to consider new terms, such as the linear term $b\partial_z\phi$ appearing in \eqref{vort}. Moreover, due to \eqref{def:zv}, the Biot-Savart law \eqref{pote}--\eqref{velo} is more complicated, and in particular, \eqref{velo} is no longer expressible as a Fourier multiplier (this was possible  in the proof of Theorem \ref{T01}), which in turn makes the use of our norm $A$ a little more difficult. 

\begin{proof}[Proof of Theorem~\ref{T02}]
For our norm, we use the same multipliers $M$ and $A$ as in the the proof of Theorem \ref{T01} (the $M$ which obeys the conditions \eqref{eq:Mconds}). 
Let $T$ be the maximal time interval $[0,T]$ such that the following estimates hold: 
\begin{subequations} \label{bst}  
\begin{align}
  \label{bst:a}
  &\Vert Af\Vert_{L^{\infty}L^{2}} 
 +
  \nu^{1/2}\Vert \nabla_LAf\Vert_{L^2L^{2}} 
 +
  \left\Vert \sqrt{-\dot{M}M}\ang{D}^Nf\right\Vert_{L^2L^{2}}
  \le8\varepsilon
  \\&
  \label{bst:b}
  \Vert u_0^z\Vert_{L^\infty L^{2}} + \nu^{1/2}\Vert \partial_v u_0^z\Vert_{L^2L^{2}}
  \le8\varepsilon. 
  \end{align}
\end{subequations}
By local well-posedness, the quantities on the left-hand side of  \eqref{bst} take values continuously in time, $T > 0$, and \eqref{bst} holds on a smaller time interval with the ``8'' replaced by a ``2''.  
The following proposition implies Theorem \ref{T02}. 
\begin{proposition} \label{prop:boot}
For $\delta$ and $\varepsilon$ chosen sufficiently small, the estimates in \eqref{bst} hold on $[0,T]$ with ``8'' replaced with ``4''. It follows by continuity that $T = +\infty$. 
\end{proposition}
We now prove Proposition \ref{prop:boot}.

\subsubsection{Energy estimate on $f$, \eqref{bst:a}}
In this subsection we improve \eqref{bst:a}. 
Applying the operator $A$ to \eqref{vort} we arrive at the energy estimate for the vorticity $f$:
  \begin{align}
  \label{EQ401}
  &\frac{1}{2}\Vert Af(T)\Vert_{L^{2}}^2
  +
  \nu\Vert \nabla_LAf(T)\Vert_{L^{2}L^{2}}^2
  +\left\Vert \sqrt{-\dot{M}M}\ang{D}^Nf(T)\right\Vert_{L^{2}L^{2}}^2
  \nonumber\\&
  \quad=
  \frac{1}{2}\Vert Af(0)\Vert_{L^{2}}^2 
  -\iint A(u\cdot\nabla_tf)Af\,dVdt
  +\iint A(b\partial_z\phi)Af\,dVdt
  \nonumber\\&
  \qquad
  +
  \nu\iint A((a^2-1)\partial^L_{vv}f)Af\,dVdt
  \nonumber\\&\quad
  =\frac{1}{2}\Vert Af(0)\Vert_{L^{2}}^2-\mathcal{T}+{\rm S}+{\rm DE}.
  \end{align}
We next bound  the {\em transport}, the {\em source}, and the {\em dissipation error} terms on the right side of \eqref{EQ401}.
  
We first consider the transport term $\mathcal{T}$. 
We begin by decomposing the velocity field into zero and non-zero modes (recalling \eqref{pote}--\eqref{velo}, and the definition of $\nabla_t$), we have 
  \begin{align}
  \mathcal{T} 
  &=  \iint A(u_0^z \partial_z f)Af\,dVdt+\iint A(\grad_t^{\perp}\Delta_t^{-1}\ne{f}\cdot\nabla_tf)Af\,dVdt \nonumber  \\
  &= \mathcal{T}_0+\ne{\mathcal{T}} \notag.
  \end{align}
Using \eqref{gb:bd} and the $H^N$ product rule ($N > 1$), the term $\ne{\mathcal{T}}$ is bounded as
  \begin{align}
  \label{EQ402}
  \abs{\ne{\mathcal{T}}}
  &\lesssim
  \Vert \nabla_t^{\perp}\Delta_t^{-1}\ne{f}\Vert_{L^{2}H^{N}}\Vert \nabla_tf\Vert_{L^{2}H^{N}}\Vert Af\Vert_{L^{\infty}L^{2}}
  \nonumber\\&
  \lesssim  \Vert\nabla_L^{\perp}\Delta_t^{-1}\ne{f}\Vert_{L^{2}H^{N}} \Vert \nabla_Lf\Vert_{L^{2}H^{N}}\Vert Af\Vert_{L^{\infty}L^{2}}.
  \end{align}
By using \eqref{eq:M:cond:3}, the fact that $|k| \geq 1$, and the Lemma \eqref{l03}, which allows us to  commute $\Delta_t^{-1}$ and $\sqrt{-\dot{M}M}$, there holds
\begin{align*}  
\Vert\nabla_L^{\perp}\Delta_t^{-1}\ne{f}\Vert_{L^{2}H^{N}}
  \lesssim
  \left\Vert\sqrt{-\dot{M}M}\Delta_L\Delta_t^{-1}\ne{f}\right\Vert_{L^{2}H^{N}}
  \lesssim
   \left\Vert\sqrt{-\dot{M}M}\ne{f}\right\Vert_{L^{2}H^{N}}.
\end{align*}
From \eqref{EQ402} and \eqref{bst}, we then obtain
\begin{align*}
  \abs{\ne{\mathcal{T}}}
  \lesssim
  \left\Vert\sqrt{-\dot{M}M}\ne{f}\right\Vert_{L^{2}H^{N}}
  \Vert \nabla_Lf\Vert_{L^{2}H^{N}}\Vert Af\Vert_{L^{\infty}L^{2}}
  \lesssim \varepsilon^3\nu^{-1/2}.
\end{align*}
This is consistent with Proposition \ref{prop:boot} provided we use the hypothesis that $\varepsilon \ll \nu^{1/2}$. 

 For the zero mode term $\mathcal{T}_0$, we begin as in ~Section~\ref{sec:Couette}. 
Similarly to \eqref{eq:T0Couette}, upon integrating by parts  we need to consider  the commutator  
\begin{align*}
  \mathcal{T}_0
  = -\iint \left(A(u_0^z \partial_z\ne{f})-u_0^z\partial_zA\ne{f}\right)A\ne{f}\,dVdt.
\end{align*}
By Plancherel's theorem,  
\begin{align*}  
-\mathcal{T}_0
  = \frac{1}{4 \pi^2}
  \sum_{k\neq0}\int_t\int_\xi\int_\eta \left(A(k,\xi)-A(k,\xi-\eta)\right)\widehat{u_0^z}(\eta)i k\hat{f}(k,\xi-\eta) A(k,\xi)\bar{\hat{f}}(k,\xi)\,d\eta d\xi dt.   
\end{align*}
The commutator is estimated as in \eqref{ineq:T0postCommCouette} above (in particular, we apply \eqref{eq:com:1} and \eqref{eq:com:2}); we also apply $u_0^z = a\partial_v\Delta_t^{-1}f_0$, 
  \begin{align*}
  \abs{\mathcal{T}_0}
  &\lesssim
  \Vert \partial_v\left(a \partial_v \Delta_t^{-1} f_0\right)\Vert_{L^{\infty}H^{N}} \Vert \ne{f}\Vert_{L^{2}H^{N}}\Vert A\ne{f}\Vert_{L^{2}L^{2}}
  \nonumber\\&
   \lesssim \left(\Vert a-1\Vert_{L^\infty H^{N+1}} \Vert \partial_{v}\Delta_t^{-1}f_0\Vert_{L^{\infty}H^{N}} +\Vert a\partial_{vv}\Delta_t^{-1}f_0\Vert_{L^{\infty}H^{N}}\right)  \Vert \ne{f}\Vert_{L^{2}H^{N}}\Vert A\ne{f}\Vert_{L^{2}L^{2}}.
  \end{align*}
Note that by considering separately frequencies less than and greater than one, there holds
\begin{align*}
\norm{\partial_v\Delta_t^{-1}f_0}_{L^\infty H^N} \leq \norm{\partial_v \Delta_t^{-1}f_0}_{L^\infty L^2} + \norm{\partial_{vv}\Delta_t^{-1}f_0}_{L^\infty H^N}. 
\end{align*}
Therefore, by \eqref{gb:bd}, \eqref{ineq:fneqdec}, Lemma \ref{l04}, and Lemma \ref{diff:to} (see also \eqref{eq:v0:u0} below), we further obtain 
  \begin{align*}
    \abs{\mathcal{T}_0}  & \lesssim \left(\Vert \partial_{v}\Delta_t^{-1}f_0\Vert_{L^{\infty}L^{2}}  + \Vert \partial_{vv}\Delta_t^{-1}f_0\Vert_{L^{\infty}H^{N}}\right)
  \Vert \ne{f}\Vert_{L^{2}H^{N}}\Vert A\ne{f}\Vert_{L^{2}L^{2}}
  \nonumber\\&
   \lesssim
  \left(\Vert \partial_{v}\Delta_t^{-1}f_0\Vert_{L^{\infty}L^{2}}
  +\Vert f_0\Vert_{L^{\infty}H^{N}}\right)
  \Vert \ne{f}\Vert_{L^{2}H^{N}}\Vert A\ne{f}\Vert_{L^{2}L^{2}}. 
\end{align*} 
Using \eqref{gb:bd} and Sobolev embedding, we have for $\delta$ sufficiently small (recall $u_0^z = -a\partial_v\Delta_t^{-1}f_0$), 
\begin{align*} 
  \abs{\mathcal{T}_0} & \lesssim  \left(\Vert u_0^z \Vert_{L^{\infty}L^{2}} +\Vert f_0\Vert_{L^{\infty}H^{N}}\right) \Vert \ne{f}\Vert_{L^{2}H^{N}}\Vert A\ne{f}\Vert_{L^{2}L^{2}}
  \nonumber\\&
  \lesssim
  \varepsilon^3\nu^{-1/3}.
  \end{align*}
Here in the last step we used the bootstrap hypotheses \eqref{bst}. 
This is now consistent with Proposition \ref{prop:boot} provided $\varepsilon \ll \nu^{1/3}$ (which of course is weaker than the hypotheses $\varepsilon \ll \nu^{1/2}$). 

Next we estimate the source term $S$ in \eqref{EQ401}. Recall that $b = b(t,v)$ and thus
  \begin{align*}
  S = \iint A(b\partial_z\Delta_t^{-1}\ne{f})A\ne{f}\,dVdt.
  \end{align*}
Via Plancherel's theorem, \eqref{eq:M:cond:3}, and \eqref{eq:M:cond:6}, there holds 
  \begin{align}
  |S|&=
\left|  \iint A(b\partial_{z}\Delta_L^{-1}\Delta_L\Delta_t^{-1}\ne{f})A\ne{f}\,dVdt \right| \notag \\
  & = \frac{1}{4\pi^2} \left| \sum_{k\neq0}\iiint A(k,\xi)\hat{b}(\xi-\eta) \frac{ik}{k^2 + \abs{\eta-kt}^2} \widehat{(\Delta_L\Delta_t^{-1}\ne{f})}(k,\eta) A(k,\xi) \bar{\hat{f}}(k,\xi) \,d\eta d\xi dt \right| \notag  \\
  &\lesssim \sum_{k\neq0}\iiint \abs{A(k,\xi)\hat{b}(\xi-\eta) \frac{-\dot{M}(k,\eta)}{M(k,\eta)} \widehat{(\Delta_L \Delta_t^{-1}\ne{f})}(k,\eta) A(k,\xi) \hat{f}(k,\xi)} \,d\eta d\xi dt \notag  \\ 
& \lesssim \sum_{k\neq0}\iiint \abs{A(k,\xi)\brak{\xi-\eta}\hat{b}(\xi-\eta) \sqrt{\frac{-\dot{M}(k,\eta)}{M(k,\eta)}} \widehat{(\Delta_L\Delta_t^{-1}\ne{f})}(k,\eta)} \notag \\ 
& \quad\quad \times \abs{\sqrt{\frac{-\dot{M}(k,\xi)}{M(k,\xi)}} A(k,\xi) \hat{f}(k,\xi)} \,d\eta d\xi dt \notag  \\ 
& \lesssim \sum_{k\neq0}\iiint \abs{\brak{\xi-\eta}^{N+1}\hat{b}(\xi-\eta) \sqrt{\frac{-\dot{M}(k,\eta)}{M(k,\eta)}}A(k,\eta)\widehat{(\Delta_L\Delta_t^{-1}\ne{f})}(k,\eta)} \notag \\ 
& \quad\quad \times \abs{\sqrt{\frac{-\dot{M}(k,\xi)}{M(k,\xi)}} A(k,\xi) \hat{f}(k,\xi)} \,d\eta d\xi dt. 
\label{eq:S:bound:1}
  \end{align}
 Young's convolution inequality followed by Lemma \ref{l03} and \eqref{gb:bd} yields 
  \begin{align*}
  |S|&
  \lesssim
  \Vert b\Vert_{H^{N+2}}\left\Vert \sqrt{-\dot{M}{M}}\ang{D}^N\Delta_L\Delta_t^{-1}\ne{f}\right\Vert_{L^{2}L^{2}} \left\Vert \sqrt{-\dot{M}{M}}\ang{D}^N\ne{f}\right\Vert_{L^{2}L^{2}}
  \nonumber\\&
  \lesssim \delta\left\Vert \sqrt{-\dot{M}M}\ang{D}^N\ne{f}\right\Vert_{L^{2}L^{2}}^2.
  \end{align*}
Hence, for $\delta$ sufficiently small, this term is absorbed on the left hand side of \eqref{EQ401}.

Finally, we estimate the diffusion error term $DE$ in \eqref{EQ401}. 
First divide into zero and non-zero modes,  
  \begin{align*}
  DE &=\nu\iint A((a^2-1)\partial_{vv}f_0)Af_0\,dVdt + \nu\iint A((a^2-1)\partial^L_{vv}\ne{f})A\ne{f}\,dVdt,\nonumber\\
  & = DE_0+\ne{DE}.
  \end{align*}
Integration by parts, followed by \eqref{g:b}, \eqref{gb:bd}, and a bound similar to \eqref{eq:S:bound:1}, implies
  \begin{align*}
  |\ne{DE}|=&
   \left| 2 \nu  \iint A(b\, \partial^L_{v}\ne{f})A\ne{f}\,dVdt 
   +\nu\iint A((a^2-1)\partial^L_{v}\ne{f})\partial_{v}^LA\ne{f}\,dVdt\right|
   \nonumber\\
   \lesssim&
   \delta\nu\Vert \nabla_L\ne{f}\Vert_{L^{2}H^{N}}\Vert A\ne{f}\Vert_{L^{2}L^{2}}
   +
   \delta\nu\Vert \nabla_L\ne{f}\Vert_{L^{2}H^{N}}\Vert \nabla_LA\ne{f}\Vert_{L^{2}L^{2}}
   \nonumber\\
   \lesssim&
   \delta\nu\Vert \nabla_L\ne{f}\Vert_{L^{2}H^{N}}^2+\delta\nu\Vert A\ne{f}\Vert_{L^{2}L^{2}}^2 \notag\\
   \lesssim&
   \delta\nu\varepsilon^2\nu^{-1}
   \lesssim \delta\varepsilon^2. 
  \end{align*}
This is consistent with Proposition \ref{prop:boot} provided $\delta$ is chosen sufficiently small (independently of $\nu$ and $\eps$). 
For the zero mode term, we similarly  integrate by parts and use \eqref{l2:g} to obtain
  \begin{align*}
  |DE_0|   &\lesssim \nu \Vert b \Vert_{L^{2}H^{N}}\Vert \partial_v{f_0}\Vert_{L^{2}H^{N}}\Vert f_0\Vert_{L^{\infty}H^{N}} +  \nu \Vert a^2-1 \Vert_{L^{\infty}H^{N}} \Vert \partial_v{f_0}\Vert_{L^{2}H^{N}}^2
  \nonumber\\&
  \lesssim 
  \delta\varepsilon^2,
  \end{align*}
which is again consistent with Proposition \ref{prop:boot} by choosing $\delta$ sufficiently small. 
This concludes improvement of \eqref{bst:a} for $\varepsilon$ and $\delta$ sufficiently small. 

\subsubsection{Velocity estimate} 
Here we improve \eqref{bst:b}. 
By Lemmas \ref{diff:back} and \ref{diff:to} we have
  \begin{align}
  \label{eq:v0:u0}
  \frac{1}{2}\Vert u_0\Vert_{L^2}
  \le
  \Vert \vv_0\Vert_{L^2}
  \le
  2\Vert u_0\Vert_{L^2},
  \end{align}
provided $\delta$ is small enough. Thus, in order to prove \eqref{bst:b}, we only need to prove the inequality in the original coordinates for $\vv_0$. 

Taking the $x$ average of \eqref{EQ01}, we obtain the equation for $\vv_0^x$
\begin{align*}
\partial_t\vv_0^x + (\vv\cdot \nabla \vv^x)_0 -\nu\Delta \vv_0^x = 0,
\end{align*}
where we used $\vv^y_0=0$ due to the divergence-free condition. 
Multiplying the  above equation by $\vv_0^x$ and integrating over $y$ on $\mathbb{R}$, we obtain the energy estimate
\begin{align*}
  \frac{1}{2}\frac{d}{dt}\Vert \vv_0^x\Vert_{L^2}^2+\nu\Vert \nabla\vv_0^x\Vert_{L^2}^2
  =
  -\int (\vv\cdot \nabla \vv^x)_0\vv^x_0\,dy.
\end{align*}
Integrating in time gives
  \begin{align*}
  \frac{1}{2}\Vert \vv_0^x(t)\Vert_{L^2}^2+\nu\Vert \nabla\vv^x_0(t)\Vert_{L^2L^2}^2
  &=
  \frac{1}{2}\Vert \vv_0^x (0)\Vert_{L^2}^2
  -\iint (\vv\cdot \nabla \vv^x)_0\vv_0^x\,dydt
  \nonumber\\&=
  \frac{1}{2}\Vert \vv_0^x(0)\Vert_{L^2}^2
  -\mathcal{\bar T}.
  \end{align*}
In order to estimate the transport term $\mathcal{\bar T}$, we recall that $\vv_0^y = 0$, use the divergence free condition, and apply Plancherel's theorem to arrive at
  \begin{align*}
  \mathcal{\bar T} & = \iint  \vv_0^x \, \partial_x \vv^x_0 \vv_0\,dydt + \iint (\vv_{\neq} \cdot \grad \vv^x_{\neq})_0\vv_0^x \,dydt \\ 
& = \iint (\vv_{\neq} \cdot \grad \vv^x_{\neq})_0\vv_0^x \,dydt   
= \iint (\grad\cdot(\vv_{\neq} \vv^x_{\neq}))_0\vv_0^x \,dydt   
= \iint \left(\partial_y(\vv_{\neq}^y \vv^x_{\neq})\right)_0\vv_0^x \,dydt.
\end{align*} 
Physically this term corresponds to the transfer of kinetic energy from the non-zero modes to the zero mode. 
By changing coordinates back to $(z,v)$ inside the $x$ integral we obtain,  
\begin{align*} 
\left(\partial_y(\vv_{\neq}^y \vv^x_{\neq})\right)_0 & = -\int a(\partial_v - t\partial_z)\left( a(\partial_v - t\partial_z)\phi_{\neq} \partial_z \phi_{\neq} \right) dz \\ 
& = -\int a\partial_v\left( a(\partial_v - t\partial_z)\phi_{\neq} \partial_z \phi_{\neq} \right) dz. 
\end{align*}
Therefore, by \eqref{gb:bd}, Lemma \ref{diff:to}, Lemma \ref{lem:DelLDelt}, and the Sobolev embedding we have 
\begin{align*}
\abs{\mathcal{\bar T}} & \lesssim \left(\norm{\partial_v \partial_v^L \phi_{\neq}}_{L^2L^2} \norm{\partial_z\phi_{\neq}}_{L^2 L^\infty} + \norm{\partial_v^L \phi_{\neq}}_{L^2 L^\infty} \norm{\partial_{vz}\phi_{\neq}}_{L^2 L^2}\right)\norm{\bar{v}_0^x}_{L^\infty L^2} \\ 
& \lesssim \norm{f_{\neq}}_{L^2 H^N}^2 \norm{\bar{v}_0^x}_{L^\infty L^2} \\
& \lesssim \varepsilon^3 \nu^{-1/3}, 
\end{align*}
which is sufficient to improve \eqref{bst:b} and hence conclude the proof of Proposition \ref{prop:boot}. 
Therefore, this also concludes the proof of Theorem \ref{T02}. 
\end{proof}

\appendix

\section{Construction and properties of the multiplier $M$}
\label{sec:app}
In this section we recall some of the technical tools regarding the Fourier multiplier $M$ from \cite{BedrossianGermainMasmoudi15a} and adapt them to our simpler setting. 
\begin{lemma} 
\label{l01}
There exists a multiplier $M$ such that the conditions \eqref{eq:M:cond:1}--\eqref{eq:M:cond:5}
hold for some constant $0<c<1$.
\end{lemma}
\begin{proof}[Proof of Lemma~\ref{l01}]
We consider a multiplier of the form $M=M_1M_2$ such that both $M_1$ and $M_2$ satisfy~\eqref{eq:M:cond:1}. We choose $M_1$ such that
for $k\neq0$ it is determined by the ODE
 \begin{align*}
  -\frac{\dot{M_1}}{M_1} & =\frac{\abs{k}}{k^2+|\xi-kt|^2} \\ 
  M_1(0,k,\eta) & = 1.
  \end{align*}
Similar multipliers appeared in \cite{Zillinger2014,BGM15I,BGM15II,BedrossianGermainMasmoudi15a}. 
The above multiplier clearly satisfies \eqref{eq:M:cond:3}.
Notice that for $k \neq 0$, there holds
\begin{align*}
\frac{k^2 + \abs{\eta-kt}^2}{k^2+|\xi-kt|^2} = \frac{k^2 + \abs{\xi-kt + \eta-\xi}^2}{k^2+|\xi-kt|^2}  \lesssim 1 + \abs{\eta-\xi}^2,  
\end{align*}
and hence \eqref{eq:M:cond:6} holds for $M_1$. 
A direct computation shows that 
  \begin{align*}
  M_1(t, k, \xi)
  =
  \exp{\left(-\int_0^t \frac{\abs{k}}{k^2+|\xi-ks|^2}\,ds\right)},
  \end{align*}
which implies \eqref{eq:M:cond:2} holds for $M_1$.
Taking the derivative of $M_1$ with respect to $\xi$ gives
  \begin{align*}
  \left|\frac{\partial_{\xi}M_1(k,\xi)}{M_1(k,\xi)}\right|
  &=
  \left|\int_0^t \frac{2\abs{k}(\xi-k s)}{(k^2+|\xi-ks|^2)^2}\,ds\right|
  \le
  \frac{2}{\abs{k}^2}\int_0^t \frac{1}{(1+|\xi/k-s|^2)}\,ds, 
  \end{align*}
which proves that \eqref{eq:M:cond:4} holds for $M_1$. Note here that $k \neq 0$ implies $|k|\geq 1$.

Next, we define $M_2$ by the differential equation (for $k \neq 0$),   
  \begin{align*}
   -\frac{\dot{M_2}}{M_2} & =\frac{\nu^{1/3}}{(\nu^{1/3}\left|t-\xi/k\right|)^{2}+1} \\ 
   M_2(0,k,\eta) & = 1. 
  \end{align*}
This multiplier was introduced in \cite{BedrossianGermainMasmoudi15a}. 
Similarly to $M_1$, we deduce that \eqref{eq:M:cond:2}, \eqref{eq:M:cond:4}, and \eqref{eq:M:cond:6} all hold for $M_2$ (and hence also for $M = M_1 M_2$). 
Since for $k\neq0$ we have that 
\begin{align*}
1\lesssim \nu^{1/3}|k,\xi-kt| \quad \mbox{ if } \nu^{-1/3}\le\left|t-\dfrac\xi k\right|
\end{align*}
and 
\begin{align*}
1\lesssim \dfrac{1}{(\nu^{1/3}\left|t-\xi/k\right|)^{2}+1} \quad \mbox{ if }  \nu^{-1/3}\ge\left|t-\dfrac\xi k\right|,
\end{align*}
inequality \eqref{eq:M:cond:5} holds for $M_2$. 
Therefore, the multiplier $M$ we constructed satisfies conditions \eqref{eq:M:cond:1}--\eqref{eq:M:cond:5}, completing the proof.
  \end{proof}
\begin{remark}
In condition \eqref{eq:M:cond:5}, the power $\nu^{-1/6}$ in front of the multiplier is sharp in the sense that $1/6$ is the smallest sacrifice we need to make to bound the expression on the right side from below by a constant. In fact, if we make $M_2$ to be positive constants independent of $\nu$ when $\left|t-\xi/k\right|\ge\nu^{-1/3}$, then the size of $\dot{M}_2$ should be approximately $\nu^{1/3}$.  
Hence, if we did not need \eqref{eq:M:cond:6} or \eqref{eq:M:cond:4}, one could construct $M_2$ to satisfy \eqref{eq:M:cond:5} with
  \begin{align*}
  & \dot{M}_2=0 \spacef {\rm if}\   \nu^{-1/3}\ge\left|t-\dfrac\xi k\right|
  \nonumber\\
  & \dot{M}_2=-\dfrac12\nu^{1/3} \spacef {\rm if}\   \nu^{-1/3}\ge\left|t-\dfrac\xi k\right|.
  \end{align*}
\end{remark}

We need the following lemma to commute $\sqrt{-\dot{M}M}$ with $\Delta_L\Delta_t^{-1}$, the latter of which is not a Fourier multiplier. 
\begin{lemma} \label{l03}
Let $f\in H^N$ and $N > 1$. Then the following estimate holds for $\delta$ sufficiently small,
\begin{align*}
  \left\Vert\sqrt{-\dot{M}M}\Delta_L\Delta_t^{-1}\ne{f}\right\Vert_{H^{N}} \lesssim \left\Vert\sqrt{-\dot{M}M}\ne{f}\right\Vert_{H^{N}}.
\end{align*}
\end{lemma}
  \begin{proof}[Proof of Lemma~\ref{l03}]
Using the equality
\begin{align*}  
\Delta_L=\Delta_t-(a^2-1)\partial_{vv}^L-b\partial_{v}^L,
\end{align*}
we have
  \begin{align}
  \label{A01}
  \left\Vert\sqrt{-\dot{M}M}\Delta_L\Delta_t^{-1}\ne{f}\right\Vert_{H^{N}}
  &\quad\lesssim
   \left\Vert\sqrt{-\dot{M}M}\ne{f}\right\Vert_{H^{N}}
   +
  \left \Vert\sqrt{-\dot{M}M}\left((a^2-1)\partial_{vv}^L\Delta_t^{-1}\ne{f}\right)\right\Vert_{H^{N}}
   \nonumber\\&
   \qquad
   +
   \left\Vert\sqrt{-\dot{M}M}\left(b\partial_{v}^L\Delta_t^{-1}\ne{f}\right)\right\Vert_{H^{N}}.
  \end{align}
By \eqref{eq:M:cond:6}, $N>1$, and \eqref{gb:bd}, we deduce
  \begin{align*}
  \left\Vert\sqrt{-\dot{M}M}\left((a^2-1)\partial_{vv}^L\Delta_t^{-1}\ne{f}\right)\right\Vert_{H^{N}}
   &=
   \left\Vert\sqrt{-\dot{M}M}\left(((a-1)^2-2(a-1))\partial_{vv}^L\Delta_t^{-1}\ne{f}\right)\right\Vert_{H^{N}}
   \nonumber\\
   &\lesssim
   \Vert a-1\Vert_{H^{N+1}}\left\Vert\sqrt{-\dot{M}M}\Delta_L\Delta_t^{-1}\ne{f}\right\Vert_{H^{N}}
   \nonumber\\
   &\lesssim
   \delta\left\Vert\sqrt{-\dot{M}M}\Delta_L\Delta_t^{-1}\ne{f}\right\Vert_{H^{N}}, 
  \end{align*}
and similarly
  \begin{align*}
  \left\Vert\sqrt{-\dot{M}M}\left(b\partial_{v}^L\Delta_t^{-1}\ne{f}\right)\right\Vert_{H^{N}}
&\lesssim
   \Vert b\Vert_{H^{N+1}}\left\Vert\sqrt{-\dot{M}M}\Delta_L\Delta_t^{-1}\ne{f}\right\Vert_{H^{N}}
 \notag\\
&\lesssim
   \delta\left\Vert\sqrt{-\dot{M}M}\Delta_L\Delta_t^{-1}\ne{f}\right\Vert_{H^{N}}.
  \end{align*}
Since $\delta\ll1$, the result follows from \eqref{A01} immediately.
\end{proof}

The following lemma is proved in the same manner as Lemma \ref{l03}, although slightly simpler. 
In particular, this lemma shows that $\Delta_L \Delta_t^{-1}$ can be approximately treated as the identity for $\delta$ sufficiently small. 
\begin{lemma}\label{lem:DelLDelt}
Let $f\in H^N$ and $N > 1$. Then the following estimate holds for $\delta$ sufficiently small,
\label{l05}
\begin{align*}
  \left\Vert\Delta_L\Delta_t^{-1}\ne{f}\right\Vert_{H^{N}}
  \lesssim
   \left\Vert\ne{f}\right\Vert_{H^{N}}. 
  \end{align*}
\end{lemma}

The following estimate applies to the zero mode of the velocity field. 
\begin{lemma}
\label{l04}
Let $f$ be such that $f_0\in H^N$, $N>1$, and $\partial_{v}\Delta_t^{-1}f_0\in L^2$. Then for $\delta$ sufficiently small, there holds 
  \begin{align*}
  \Vert\partial_{vv}\Delta_t^{-1}f_0\Vert_{H^{N}}
  \lesssim
  \Vert f_0\Vert_{H^N}+\delta\Vert\partial_{v}\Delta_t^{-1}f_0\Vert_{L^{2}}. 
  \end{align*}
\end{lemma}
\begin{proof}[Proof of Lemma~\ref{l04}]
Since 
  \begin{align*}
  \partial_{vv}\Delta_t^{-1}f_0
  =
  (\Delta_t-(a^2-1)\partial_{vv}-b\partial_v)\Delta_t^{-1}f_0,
  \end{align*}
we have by $N>1$ and \eqref{gb:bd} (interpolating $H^N$ between $L^2$ and $H^{N+1}$), 
  \begin{align*}
  \Vert\partial_{vv}\Delta_t^{-1}f_0\Vert_{H^{N}}
  &\lesssim
  \Vert f_0\Vert_{H^{N}}+\delta\Vert\partial_{vv}\Delta_t^{-1}f_0\Vert_{H^{N}}
  +\delta\Vert\partial_{v}\Delta_t^{-1}f_0\Vert_{H^{N}}
  \nonumber\\&
  \lesssim
  \Vert f_0\Vert_{H^{N}}+\delta\Vert\partial_{vv}\Delta_t^{-1}f_0\Vert_{H^{N}} +\delta\Vert\partial_{v}\Delta_t^{-1}f_0\Vert_{L^{2}}.
  \end{align*}
For $\delta\ll1$, we may absorb the second term on the right side and we obtain the desired result.
\end{proof}


\section*{Acknowledgments}
The work of JB was partially supported by NSF grant DMS-1462029, and by an Alfred P. Sloan Research Fellowship. The work of VV was partially supported by NSF grant DMS-1514771 and by an Alfred P. Sloan Research Fellowship. The work of FW was partially supported by NSF grant DMS-1514771.


\def\cprime{$'$}


\end{document}